\documentclass[reqno,12pt]{amsart}

\usepackage{amssymb,amsthm,amsmath}
\usepackage[left=2.5cm,right=2.5cm]{geometry}
\usepackage{stackrel}

\def\Med{{\rm Med\,}}
\def\Cov{{\rm Cov\,}}
\def\cA{\mathcal{A}}

\def\E{\mathbb E}
\def\p{\mathbb P}

\newcommand*{\ind}[1]{\mathbf{1}_{\{#1\}}}
\newcommand*{\Ind}[1]{\mathbf{1}_{#1}}

\newcommand{\ub}[1]{^{(#1)}}

\newcommand{\R}{\mathbb{R}}
\newcommand{\N}{\mathbb{N}}
\newcommand{\id}{\mathrm{Id}}
\newcommand{\tr}{\mathrm{tr}}

\newtheorem{lemma}{Lemma}[section]

\newtheorem{theorem}[lemma]{Theorem}

\newtheorem{defi}[lemma]{Definition}

\title[A note on the Hanson-Wright inequality]{A note on the Hanson-Wright inequality for random vectors with dependencies}

\author{Rados{\l}aw Adamczak}
\address{Institute of Mathematics, University of Warsaw, ul. Banacha 2,
  02-097 Warszawa, POLAND.}

\thanks{Institute of Mathematics, University of Warsaw, e-mail: R.Adamczak@mimuw.edu.pl. Research partially supported by the NCN grant no. 2012/05/B/ST1/00412}

\begin{document}

\begin{abstract}
We prove that quadratic forms in isotropic random vectors $X$ in $\R^n$, possessing the convex concentration property with constant $K$, satisfy the Hanson-Wright inequality with constant $CK$, where $C$ is an absolute constant, thus eliminating the logarithmic (in the dimension) factors in a recent estimate by Vu and Wang. We also show that the concentration inequality for all Lipschitz functions implies a uniform version of the Hanson-Wright inequality for suprema of quadratic forms (in the spirit of the inequalities by Borell, Arcones-Gin\'{e} and Ledoux-Talagrand). Previous results of this type relied on stronger isoperimetric properties of $X$ and in some cases provided an upper bound on the deviations rather than a concentration inequality.

In the last part of the paper we show that the uniform version of the Hanson-Wright inequality for Gaussian vectors can be used to recover a recent concentration inequality for empirical estimators of the covariance operator of $B$-valued Gaussian variables due to Koltchinskii and Lounici.
\end{abstract}


\maketitle
\section{Introduction}

The Hanson-Wright inequality asserts that if $X_1,\ldots,X_n$ are independent mean zero, variance one random variables with sub-Gaussian tail decay, i.e. such that
for all $t > 0$,
\begin{displaymath}
\p(|X_i|\ge t) \le 2 \exp(-t^2/K^2),
\end{displaymath}
and $A = [a_{ij}]_{i,j=1}^n$ is an $n\times n$ matrix, then the quadratic form
\begin{displaymath}
Z = \sum_{i,j=1}^n a_{ij}X_i X_j
\end{displaymath}
satisfies the inequality
\begin{displaymath}
\p(|Z - \tr A| \ge t) \le 2\exp\Big(-\min\Big(\frac{t^2}{CK^4\|A\|_{HS}^2},\frac{t}{CK^2\|A\|}\Big)\Big)
\end{displaymath}
for all $t > 0$, where $C$ is a universal constant.
Here and in what follows $\|A\|_{HS} = (\sum_{i,j\le n} a_{ij}^2)^{1/2}$ is the Hilbert-Schmidt norm of $A$, whereas $\|A\| = \sup_{|x|\le 1} |Ax|$ is the operator norm of $A$ ($|\cdot|$ denotes the standard Euclidean norm in $\R^n$).
Actually Hanson and Wright \cite{HansonWright} proved a somewhat weaker inequality in which $\|A\|$ was replaced by the operator norm of the matrix $\tilde{A} = [|a_{ij}|]_{i,j=1}^n$. The original argument worked also only for symmetric random variables, the general mean zero case was proved by Wright in \cite{MR0353419}. The above version with the operator norm of $A$ appeared in many works under different sets of assumptions. For Gaussian variables it follows from estimates for general Banach space valued polynomials by Borell \cite{Bo} and Arcones-Gin\'e \cite{AG}. Independent proofs were also provided by Ledoux-Talagrand \cite{LT} and Lata{\l}a \cite{L1,L2}. It is also well known that the general case can be reduced to the Gaussian one by comparison of moments or a decoupling and contraction approach \cite{MR2322114,MR3096532,nonlipschitz,MR3125258}. As observed by Lata{\l}a \cite{L1} in the Gaussian case the Hanson-Wright inequality can be reversed (up to universal constants). Lata{\l}a provided also two-sided moment and tail inequalities for higher degree homogeneous forms in Gaussian variables \cite{L2} (see also \cite{nonlipschitz}).

The interest in Hanson-Wright type estimates has been recently revived in connection with non-asymptotic theory of random matrices and related statistical problems \cite{RSA:RSA20561,MR2880281}. Since in many applications one considers quadratic forms in random vectors with dependencies among coefficients, some recent work has been devoted to proving counterparts of the Hanson-Wright inequality in a dependent setting. In particular in \cite{MR2994877} a corresponding upper tail inequality is proved for positive definite matrices and sub-Gaussian random vectors $X$ (we recall that a random vector $X$ in $\R^n$ is sub-Gaussian with constant $K$ if for all $u \in S^{n-1}$, and all $t > 0$, $\p(|\langle X, u\rangle| \ge t) \le 2\exp(-t^2/K^2)$).  It is easy to see that in this setting one cannot hope for a lower tail estimate as a sub-Gaussian random vector can vanish with probability separated from zero. In \cite{RSA:RSA20561}, Vu and Wang consider vectors satisfying the convex concentration property (see Definition \ref{def:convex concentration} below) and prove that if $X$ is a random vector in $\R^n$ in the isotropic position (i.e. with mean zero and covariance matrix equal to identity) which has the convex concentration property with constant $K$, then for all $t > 0$,
\begin{align}\label{eq:Vu-Wang}
\p(|X^TAX - \tr A| \ge t) \le C\log n\exp\Big(-CK^{-2}\min\Big(\frac{t^2}{\|A\|_{HS}^2\log n},\frac{t}{\|A\|}\Big)\Big).
\end{align}
(We remark that Vu and Wang considered complex random vectors with complex conjugate-transpose operation instead of transpose, but since we are interested here primarily in the real case, we do not state their result in this version. In fact it is not difficult to pass from the real version to the complex one).

One of the objectives of this paper is to remove the dependence on dimension in the above estimate (Theorem \ref{thm:Hanson-Wright} below) as well as to prove corresponding uniform estimates for suprema of quadratic forms under some stronger assumptions on the random vector $X$ (Theorem \ref{thm:uniform_Hanson-Wright}). Such uniform versions (corresponding to Banach space valued quadratic forms) for Gaussian random vectors were considered e.g. by Borell \cite{Bo} and Arcones-Gin\'e \cite{AG}, whereas the Rademacher case was studied by Talagrand \cite{TalagrandNewConc} and Bousquet-Boucheron-Lugosi-Massart \cite{BBLM}.
In Theorem \ref{thm:uniform_Hanson-Wright} we prove that a uniform estimate is a consequence of the concentration property for Lipschitz functions.

The estimates provided by uniform Hanson-Wright inequalities are expressed in terms of expectations of suprema of certain empirical processes. Since estimating such expectations is in general difficult, direct applications of such inequalities are limited. In our last result, Theorem \ref{thm:Koltchinskii-Lounici} presented in Section \ref{sec:covariance}, we provide one example in which it is possible to effectively bound the empirical process involved in the estimate, i.e. we recover a recent concentration result for empirical approximations of the covariance operator for Banach space valued Gaussian variables, obtained first by Koltchinskii and Lounici by other methods \cite{2014arXiv1405.2468K}.

\medskip

The organization of the paper is as follows. In the next section we present our main results together with some additional discussion. Next, in Section \ref{sec:proofs} we provide proofs. Finally, in Section \ref{sec:covariance} we present the aforementioned application of uniform estimates for quadratic forms.

\medskip

\paragraph{\bf Acknowledgements}  The author would like to thank Vladimir Koltchinskii and Karim Lounici for interesting conversations during The Seventh International Conference on High Dimensional Probability. The results of this paper grew directly out of those conversations. Separate thanks go to the organizers of the conference.

\section{Main results\label{sec:main_results}}
To introduce the setting for our estimates let us first recall the standard definitions of concentration properties of random vectors.

\begin{defi}[Concentration property]\label{def:concentration}
Let $X$ be a random vector in $\R^n$. We will say that $X$ has the concentration property with constant $K$ if for every 1-Lipschitz function $\varphi \colon \R^n \to \R$, we have $\E |\varphi(X)| < \infty$ and for every $t > 0$,
\begin{displaymath}
\p(|\varphi(X) - \E\varphi(X)| \ge t) \le 2\exp(-t^2/K^2).
\end{displaymath}
\end{defi}

The concentration property of random vectors has been extensively studied in the recent forty years, starting with the celebrated results by Borell \cite{BorellGaussianConc} and
Sudakov-Tsirelson \cite{SCGaussianConc} who established it for Gaussian measures. Many efficient techniques for proving concentration have been discovered, including e.g. isoperimetric techniques, functional inequalities, transportation of measure, semigroup tools. We refer to the monograph \cite{LedouxConcBook} by Ledoux for a thorough discussion of this topic.

\begin{defi}[Convex concentration property]\label{def:convex concentration}
Let $X$ be a random vector in $\R^n$. We will say that $X$ has the convex concentration property with constant $K$ if for every 1-Lipschitz convex function $\varphi \colon \R^n \to \R$, we have $\E |\varphi(X)| < \infty$ and for every $t > 0$,
\begin{displaymath}
\p(|\varphi(X) - \E\varphi(X)| \ge t) \le 2\exp(-t^2/K^2).
\end{displaymath}
\end{defi}

\paragraph{\bf Remarks}
\paragraph{1.} The convex concentration property has been first observed by Talagrand, who proved it for the uniform measure on the discrete cube \cite{TalCube} and for general product measures with bounded support \cite{MR1361756} by means of his celebrated convex distance inequality. In the non-product case it has been obtained by Samson \cite{Samson} for vectors satisfying some uniform mixing properties and recently by Paulin \cite{Paulin} under Dobrushin type criteria. From Talagrand's results it also follows that the convex concentration property is satisfied by vectors obtained via sampling without replacement \cite{Paulin, 2014arXiv1402.3660A}. Sub-Gaussian estimates for the upper tails of Lipschitz functions of product random vectors were also obtained by Ledoux \cite{MR1399224} and later Adamczak in the unbounded case \cite{AdLogSobConv} by means of log-Sobolev inequalities.

\paragraph{2.} Note that the convex concentration property is preserved if we replace $X$ with $UX+b$,where $U$ is a deterministic orthogonal matrix and $b\in \R^n$.

Our first result is the following

\begin{theorem}\label{thm:Hanson-Wright}
Let $X$ be a mean zero random vector in $\R^n$. If $X$ has the convex concentration property with constant $K$ then
for any $n\times n$ matrix $A = [a_{ij}]_{i,j=1}^n$ and every $t> 0$,
\begin{align}\label{eq:Hanson-Wright}
\p(|X^T A X - \E (X^T A X)| \ge t) &\le 2\exp\Big(-\frac{1}{C}\min\Big(\frac{t^2}{K^2\|A\|_{HS}^2\|\Cov(X)\|},\frac{t}{K^2\|A\|}\Big)\Big)\\
&\le 2\exp\Big(-\frac{1}{C}\min\Big(\frac{t^2}{2K^4\|A\|_{HS}^2},\frac{t}{K^2\|A\|}\Big)\Big).\nonumber
\end{align}
for some universal constant $C$.
\end{theorem}

\paragraph{\bf Remarks}

\paragraph{1.} The above theorem improves the estimate \eqref{eq:Vu-Wang} due to Vu-Wang by removing the dimension dependent factors (note that in the isotropic case $\E X^TAX = \tr A$ and $\|\Cov(X)\| = \|\id\| = 1$).

\paragraph{2.} The assumption that $X$ is centered is introduced just to simplify the statement of the theorem. Note that if $X$ has the convex concentration property  with constant $K$, then so does $\tilde{X} = X - \E X$. Moreover, a quadratic form in $X$ can be decomposed into a sum of a quadratic form in $\tilde{X}$ and an affine function of $X$. Since linear functions are convex, Lipschitz, their deviations can be controlled by the convex concentration property.
We leave the precise formulation of the corresponding inequality to the Reader.
\paragraph{3.} As it will become clear from the proof, similar theorems hold if instead of sub-Gaussian concentration inequality for convex functions one assumes some other rate of decay for the tail probabilities. The whole argument remains then valid, one just has to modify accordingly the right-hand side of \eqref{eq:Hanson-Wright}. Convex concentration property with sub-exponential tail decay was studied e.g. in \cite{MR1701522}.
\paragraph{4.} We remark that it is not true that if $X = (X_1,\ldots,X_n)$ where $X_i$ are i.i.d. sub-Gussian random variables, then $X$ has the convex concentration property with a constant independent of dimension (as noted in \cite{AdLogSobConv} following \cite{MR1491527}). Therefore, Theorem \ref{thm:Hanson-Wright} does not imply the standard Hanson-Wright inequality.

Our second result concerns a uniform version of the Hanson-Wright inequality for suprema of quadratic forms and is contained in the following

\begin{theorem}\label{thm:uniform_Hanson-Wright}
Let $X$ be a mean zero random vector in $\R^n$. Assume that $X$ has the concentration property with constant $K$. Let $\cA$ be a bounded set of $n\times n$ matrices and consider the random variable
\begin{displaymath}
Z = \sup_{A\in \cA} \Big(X^T AX - \E X^T AX\Big).
\end{displaymath}
Then, for every $t > 0$,
\begin{align}\label{eq:uniform_Hanson-Wright}
\p(|Z - \E Z| \ge t) \le 2\exp\Big(-\frac{1}{C}\min\Big(\frac{t^2}{K^2 \|X\|_\cA^2 },\frac{t}{K^2\sup_{A\in \cA}\|A\|}\Big)\Big),
\end{align}
where
\begin{align*}
\|X\|_{\cA} = \E\sup_{A = [a_{ij}]_{i,j=1}^n \in \cA} |(A+A^T)X|
\end{align*}
and $C$ is a universal constant.
\end{theorem}

\paragraph{\bf Remarks}

\paragraph{1.} One can easily see that if $\cA = \{A\}$, then $\|X\|_\cA \le 2\|A\|_{HS}\sqrt{\|\Cov X\|}$. If in addition $X$ has the convex concentration property with constant $K$, then $\|\Cov X\| \le 2K^2$ (see the proof of Theorem \ref{thm:Hanson-Wright} below). Thus the conclusion of the above theorem is stronger than that of Theorem \ref{thm:Hanson-Wright}. On the other hand the assumption is also stronger. We do not know if \eqref{eq:uniform_Hanson-Wright} is implied just by the convex concentration property. This is the case if instead of $\sup_{A\in \cA} X^TAX$ one considers $\sup_{A\in \cA} X^TAY$, where $Y$ is an independent copy of $X$ (see \cite{AdLogSobConv}).

\paragraph{2.} As mentioned in the Introduction, inequalities similar to \eqref{eq:uniform_Hanson-Wright} have been proven by many authors under various sets of assumptions. In particular Borell \cite{Bo} and Arcones-Gin\'e \cite{AG} obtained inequalities for Banach space valued polynomials in Gaussian random variables. When specialised to quadratic forms, these inequalities give an upper bound on $\p(\sup_{A\in \cA} |X^T AX |\ge M + t)$, where $M$ is a certain quantile of $\sup_{A\in \cA} |X^T AX|$. The proofs are based on the Gaussian isoperimetric inequality. We do not see how to adapt their arguments to get concentration around the mean rather then deviation above a multiple of the mean. Talagrand \cite{TalagrandNewConc} proved a concentration inequality for suprema of quadratic forms in Rademacher variables, which via the Central Limit Theorem implies the concentration inequality in the Gaussian case. The upper bound in Talagrand's inequality was later generalized to higher order forms by Boucheron,Bousquet, Lugosi and Massart \cite{BBLM}.

\section{Proofs of the main results \label{sec:proofs}}
In what follows the letter $C$ will denote an absolute constant, the value of which may change between various occurrences (even in the same line).

\medskip
In the proofs we will need the following standard lemmas.

\begin{lemma}\label{le:quantiles}
Assume that a random variable $Z$ satisfies
\begin{displaymath}
\p(|Z - \E Z| \ge t) \le 2\exp(-t^2/K^2)
\end{displaymath}
for all $t > 0$. Consider $p \in (0,1)$ and let  $q_p Z = \inf\{t\in \R \colon \p(Z \le t) \ge p\}$ be the smallest $p$-th quantile of $Z$. Then
\begin{displaymath}
q_pZ \ge \E Z - K\sqrt{\log(2/p)}.
\end{displaymath}
\end{lemma}

\begin{proof}
Assume that $q_pZ < \E Z - K\sqrt{\log(2/p)}$. Then
\begin{displaymath}
\p(Z \le q_pZ ) < 2\exp(-K^2\log(2/p)/K^2) = p,
\end{displaymath}
which contradicts the standard inequality $\p(Z \le q_p Z)\ge p$.
\end{proof}

\begin{lemma}\label{le:median}
Assume that a random variable $Z$ satisfies
\begin{displaymath}
\p(|Z - \Med Z| \ge t) \le 2\exp\Big(-\min\Big(\frac{t^2}{a^2},\frac{t}{b}\Big)\Big)
\end{displaymath}
for all $t > 0$, where $\Med Z$ is a median of $Z$. Then for some absolute constant $C$ and all $t > 0$,
\begin{align}\label{eq:median}
\p(|Z - \E Z| \ge t) \le 2\exp\Big(-\frac{1}{C}\min\Big(\frac{t^2}{a^2},\frac{t}{b}\Big)\Big).
\end{align}
\end{lemma}

\begin{proof}
We have
\begin{displaymath}
|\E Z - \Med Z| \le \E|Z-\Med Z| \le 2\int_0^\infty\exp\Big(-\min\Big(\frac{t^2}{a^2},\frac{t}{b}\Big)\Big)dt \le \sqrt{\pi}a + 2b.
\end{displaymath}
Thus for $t > 2\sqrt{\pi}a + 4b$, we have
\begin{displaymath}
\p(|Z-\E Z| \ge t) \le \p(|Z-\Med Z| \ge t/2) \le 2\exp\Big(-\min\Big(\frac{t^2}{4a^2},\frac{t}{2b}\Big)\Big).
\end{displaymath}
On the other hand, there exists an absolute constant $C$, such that for $t \le 2\sqrt{\pi}a + 4b \le 8\max(a,b)$
\begin{displaymath}
\frac{1}{C}\min\Big(\frac{t^2}{a^2},\frac{t}{b}\Big) \le \log 2,
\end{displaymath}
which implies that \eqref{eq:median} is trivially satisfied.
This ends the proof of the lemma.
\end{proof}

Another simple fact we will need is

\begin{lemma}\label{le:quantile_bypass}
Let $S$ and $Z$ be random variables and $a,b,t > 0$ be such that for all $s > 0$,
\begin{align}\label{eq:conc_S}
\p(|S-\E S| \ge s) \le 2\exp(-s^2/(a + \sqrt{bt})^2)
\end{align}
and
\begin{align}\label{eq:coincidence}
\p(S\neq Z) \le 2\exp(-t/b).
\end{align}
Then
\begin{displaymath}
\p(|Z - \Med Z| \ge t) \le 2 \exp\Big(-\frac{1}{C}\min\Big(\frac{t^2}{a^2},\frac{t}{b}\Big)\Big).
\end{displaymath}
\end{lemma}

\begin{proof}
Set
\begin{displaymath}
M = a + \sqrt{bt}.
\end{displaymath}
Assume first that $t > \max(3b,2M\sqrt{\log 8})$. We then have $\p(S\neq Z) \le 1/4$ and so $\p(S \le \Med Z) \ge 1/4$, which means that $\Med Z \ge q_{1/4} S$, where $q_p S= \inf\{t\colon \p(S \le t) \ge p\}$.
By Lemma \ref{le:quantiles},
$\Med Z \ge q_{1/4} S \ge \E S - M\sqrt{\log 8}$ and thus
\begin{displaymath}
\p(Z - \Med Z \ge t) \le \p(Z\neq S) + \p(S- \E S  \ge t - M\sqrt{\log 8}) \le \p(S\neq Z) + \p(S- \E S  \ge t/2).
\end{displaymath}
Using \eqref{eq:conc_S} with $s = t/2$ and \eqref{eq:coincidence}, we obtain
\begin{displaymath}
\p(Z - \Med Z \ge t) \le 2\exp(-\frac{t}{b}) + 2\exp(-\frac{t^2}{4M^2}).
\end{displaymath}
Similarly, by replacing $S,Z$, with $-S,-Z$ and using the fact that $-\Med Z$ is a median for $-Z$, we obtain
\begin{displaymath}
\p(Z - \Med Z \le -t) \le 2\exp(-\frac{t}{b}) + 2\exp(-\frac{t^2}{4M^2}).
\end{displaymath}
Thus we have obtained that if $t > \max(3b,2M\sqrt{2\log 8})$, then
\begin{align*}
\p(|Z - \Med Z| \ge  t) &\le 4\exp(-\frac{t}{b}) + 4\exp(-\frac{t^2}{4M^2})\\
&\le 2 \exp\Big(-\frac{1}{C}\min\Big(\frac{t^2}{a^2},\frac{t}{b}\Big)\Big),
\end{align*}
where the last inequality follows by the definition of $M$ and simple calculations. This ends the proof in the case $t > \max(3b,2M\sqrt{\log 8})$.

Note that for $t \le \max(3b,2M\sqrt{\log 8})$, we have
\begin{displaymath}
 \exp(-\frac{t^2}{4M^2}) \ge \frac{1}{8} \;\textrm{or}\; \exp\Big(-\frac{t}{b}\Big) \ge \frac{1}{27},
\end{displaymath}
so trivially
\begin{displaymath}
\p(|Z - \Med Z| \ge  t) \le 27\exp(-\min\Big(\frac{t^2}{4M^2},\frac{t}{b}\Big)\Big) \le 2 \exp\Big(-\frac{1}{C}\min\Big(\frac{t^2}{a^2},\frac{t}{b}\Big)\Big).
\end{displaymath}
\end{proof}

\begin{proof}[Proof of Theorem \ref{thm:Hanson-Wright}]
Since $X^T A X =  X^T(\frac{1}{2}(A+A^T))X$, we can assume that $A$ is symmetric. Thus there exists an orthogonal matrix $U$, such that
$D = U^T X U$ is a diagonal matrix, with diagonal entries $\lambda_1,\ldots,\lambda_n$. Let $Y = UX$ and note that $Y$ also has the convex concentration property with constant $K$. Moreover $XAX^T = Y^T DY$. Thus our goal is to prove that for $t > 0$,
\begin{displaymath}
\p(|Y^T D Y - \E Y^T D Y| \ge t) \le 2\exp\Big(-\frac{1}{CK^2}\min\Big(\frac{t^2}{\|A\|_{HS}^2\|\Cov X\|},\frac{t}{\|A\|}\Big)\Big).
\end{displaymath}
Observe that $\|A\|_{HS}^2 = \sum_{i\le n} \lambda_i^2$ and $\|A\| = \max_{i\le n}|\lambda_i|$.

Let $Y_1,\ldots,Y_n$ be the coordinates of $Y$. We have $Y^T A Y = \sum_{i=1}^n \lambda_i Y_i^2 = \sum_{i=1}^n \lambda_i\ind{\lambda_i > 0}  Y_i^2 + \sum_{i=1}^n \lambda_i\ind{\lambda_i < 0} Y_i^2$
and thus, by the triangle inequality, to demonstrate the theorem it is enough to prove that for every sequence $\mu_1,\ldots,\mu_n$ of nonnegative numbers, we have
\begin{align}\label{eq:to_prove}
\p(|\sum_{i=1}^n \mu_i Y_i^2 - \E \sum_{i=1}^n \mu_iY_i^2| \ge t) \le 2\exp\Big(-\min\Big(\frac{t^2}{CK^2\|\Cov(X)\|\sum_{i=1}^n \mu_i^2},\frac{t}{CK^2\max_{i\le n} \mu_i}\Big)\Big).
\end{align}
Note that for any unit vector $u$, $\langle u,X\rangle$ is a 1-Lipschitz convex function of $X$. Since we also have $\E\langle u,X\rangle = 0$, by the convex concentration property, we get
\begin{displaymath}
u^T \Cov(X) u = \E \langle u, X\rangle^2 = 2\int_0^\infty t \p(\langle u, X\rangle \ge t)dt \le 4\int_0^\infty te^{-t^2/K^2}dt = 2K^2.
\end{displaymath}
This shows that $\|\Cov X\| \le 2K^2$.

Moreover $Y_i = \langle u, X\rangle$, where $u$ is the first row of $U$. Since $u$ is a unit vector, we get in particular
\begin{align}\label{eq:bound_on_EY2}
\E Y_i^2 \le \|\Cov(X)\| \le 2K^2.
\end{align}

Let $\varphi(y) = \sum_{i=1}^n \mu_i^2 y_i^2$ and note that $\nabla \varphi(y) = (2\mu_1y_1,\ldots,2\mu_n y_n)$. Define
\begin{align*}
B &= \{y\in \R^n \colon |\nabla \varphi(y)| \le \sqrt{\E|\nabla \varphi(Y)|^2} + \sqrt{t\max_{i\le n}\mu_i}\} \\
&=
\Big\{y\in \R^n \colon \sqrt{\sum_{i=1}^n \mu_i^2 Y_i^2} \le \sqrt{\sum_{i=1}^n \mu_i^2 \E Y_i^2}+ \frac{1}{2}\sqrt{t\max_{i\le n}\mu_i}\Big\}.
\end{align*}
By the convex concentration property of $Y$ and the fact that the function $y \mapsto \sqrt{\sum_{i=1}^n \mu_i^2 y_i^2}$ is convex and $(\max_{i\le n} \mu_i)$-Lipschitz, we get
\begin{align}\label{eq:setB}
\p(Y \notin B) \le 2\exp(-\frac{t}{4K^2 \max_{i\le n} \mu_i}).
\end{align}

Define now a new function $f\colon \R^n \to \R$ with the formula
\begin{displaymath}
f(y) = \max_{x \in B} (\langle \nabla \varphi(x),y-x\rangle + \varphi(x)).
\end{displaymath}

Note that $f$ is a convex function, moreover for $y,z \in \R^n$,
\begin{align*}
f(y) - f(z) &= \max_{x \in B} (\langle \nabla \varphi(x),y-x\rangle + \varphi(x)) - \max_{x \in B}( \langle \nabla \varphi(x),z-x\rangle + \varphi(x))\le \max_{x \in B} \langle \nabla \varphi(x),y-z\rangle \\
& \le \max_{x \in B}|\nabla \varphi(x)| |y-z| \le M|y-z|,
\end{align*}
where $M = \sqrt{\sum_{i=1}^n \mu_i^2 \E Y_i^2}+ \frac{1}{2}\sqrt{t\max_{i\le n}\mu_i}$.
Thus $f$ is convex and $M$-Lipschitz and so for all $s > 0$,
\begin{align}\label{eq:concentration_f}
\p(|f(Y) - \E f(Y)| \ge s) \le 2\exp(-s^2/K^2M^2).
\end{align}
Moreover, by convexity of $\varphi$, we have $f(y) \le \varphi(y)$ and thus for $y \in B$, we have $f(y) = \varphi(y)$.

Thanks to \eqref{eq:setB} and \eqref{eq:concentration_f} we can now apply Lemma \ref{le:quantile_bypass} with $Z = \varphi(Y)$, $S = f(Y)$, $a = K\sqrt{\sum_{i=1}^n \mu_i^2 \E Y_i^2}$ and $b = 4K^2\max_{i\le n}\mu_i$, we obtain
\begin{align*}
\p(|\varphi(Y) - \Med \varphi(Y)| \ge t) &\le 2\exp\Big(-\min\Big(\frac{t^2}{CK^2\sum_{i=1}^n \mu_i^2 \E Y_i^2},\frac{t}{CK^2\max_{i\le n}\mu_i}\Big)\Big)\\
&\le 2\exp\Big(-\min\Big(\frac{t^2}{CK^2\|\Cov (X)\|\sum_{i=1}^n \mu_i^2},\frac{t}{CK^2\max_{i\le n}\mu_i}\Big)\Big)\\
&\le 2\exp\Big(-\min\Big(\frac{t^2}{2CK^4\sum_{i=1}^n \mu_i^2},\frac{t}{CK^2\max_{i\le n}\mu_i}\Big)\Big),
\end{align*}
where in the two last inequalities we used \eqref{eq:bound_on_EY2}.

Since the above inequality holds for arbitrary $t > 0$, Lemma \ref{le:median} gives \eqref{eq:to_prove}, which ends the proof.
\end{proof}

\begin{proof}[Proof of Theorem \ref{thm:uniform_Hanson-Wright}] By the boundedness assumption on the set $\cA$ and the integrability assumption on $X$ we can assume that the set $\cA$ is finite. Let thus $\cA = \{A\ub{1},\ldots,A\ub{m}\}$, where $A\ub{k}=[a\ub{k}_{ij}]_{i,j\le n}$. Denote also $a\ub{k} = \E X^TA\ub{k}X$ and define the function $f\colon \R^n\to \R$ with the formula
\begin{align}\label{eq:def_f}
f(x) = \max_{k\le m} (x^T A\ub{k} x - a\ub{k}).
\end{align}
Note that $f$ is locally Lipschitz, moreover as the set of roots of a non-zero multivariate polynomial is of Lebesgue measure zero, for every $x$ outside
a set of Lebesgue measure zero, there exists unique $k \le m$, such that
\begin{displaymath}
f(x) = x^T A\ub{k} x - a\ub{k}.
\end{displaymath}
For $k \le m$ let $B_k$ be the set of points $x \in R^n$ such that $k$ is the unique maximizer in \eqref{eq:def_f}. Then $\R^n\setminus(\bigcup_{k\le m} B_k)$ has Lebesgue measure equal to zero, moreover the sets $B_k$ are open. Thus, we have Lebesgue-a.e.
\begin{align*}
\nabla f(x) &= \sum_{k\le m} \Ind{B_k}\Big(2 a\ub{k}_{ii}x_i  + \sum_{j\le n, j\neq i} a\ub{k}_{ij} x_j + \sum_{j\le n,j\neq i} a\ub{k}_{ji} x_j\Big)_{i=1}^n\\
& = \sum_{k\le m} \Ind{B_k}\Big(\sum_{j\le n} a\ub{k}_{ij} x_j + \sum_{j\le n} a\ub{k}_{ji} x_j\Big)_{i=1}^n
\end{align*}
and consequently
\begin{align*}
|\nabla f(x)| &= \sum_{k\le m}\Ind{B_k}\Big(\sum_{i\le n}\Big(\sum_{j\le n} a\ub{k}_{ij} x_j + \sum_{j\le n}a\ub{k}_{ji}x_j\Big)^2\Big)^{1/2}\\
&\le \max_{k\le m} \Big(\sum_{i\le n}\Big(\sum_{j\le n} a\ub{k}_{ij} x_j + \sum_{j\le n}a\ub{k}_{ji}x_j\Big)^2\Big)^{1/2}\\
& = \max_{A\in \cA} |(A+A^T)x|.
\end{align*}

Let now $B = \{x\in \R^n \colon \max_{A\in \cA} |(A+A^T)x| < \|X\|_\cA + \sqrt{t\max_{A\in \cA}\|A\|}\}$ and note that $B$ is an open convex set. Let $\lambda_k$ denote the Lebesgue measure on $\R^k$. By the Fubini theorem, the preceding discussion concerning the differentiability of $f$, the definition of the set $B$ and its convexity, for $\lambda_{2n}$ almost all pairs $(x,y) \in B\times B$ we have
\begin{displaymath}
\lambda_1(\{t\in [0,1]\colon \nabla f(tx+(1-t)y)\;\textrm{exists and}\; |\nabla f(tx+(1-t)y)| \le \|X\|_\cA + \sqrt{t\max_{A\in \cA}\|A\|}\}) = 1.
\end{displaymath}
Since $t \mapsto f(tx+(1-t)y)$ is locally Lipschitz and thus absolutely continuous, we have for such $x,y$,
\begin{align*}
f(x) - f(y) &= \int_0^1 \frac{d}{dt}f(tx+(1-t)y) dt = \int_0^1  \langle \nabla f(tx+(1-t)y),x-y\rangle dt \\
&\le (\|X\|_\cA + \sqrt{t\max_{A\in \cA}\|A\|}) |x-y|.
\end{align*}
By continuity and density arguments, the above inequality clearly extends to all $x,y \in B$, allowing us to conclude that $f$ is $M$-Lipschitz on $ B$ with $M = \|X\|_\cA + \sqrt{t\max_{A\in \cA}\|A\|}$. Let now $g\colon \R^n \to \R$ be any $M$-Lipschitz function, which coincides with $f$ on $B$ (it exists by McShane's lemma, see e.g. Lemma 7.3. in \cite{MR2976521}).
By the concentration property of $X$ we have for all $s > 0$,
 \begin{displaymath}
\p(|g(X)- \E g(X)| \ge s) \le 2\exp(-s^2/K^2M^2)
\end{displaymath}
and
\begin{displaymath}
\p(X \notin  B) = \p\Big(\max_{A\in \cA} |(A+A^T)x|\ge \|X\|_\cA + \sqrt{t\max_{A\in \cA}\|A\|}\Big) \le 2 \exp(-t/4K^2\max_{A\in\cA}\|A\|),
\end{displaymath}
where we used that the function $x\mapsto \max_{A\in\cA} |(A+A^T)x|$ has the Lipschitz constant bounded by $\max_{A \in \cA}\|A+A^T\| \le 2\max_{A\in\cA}\|A\|$.
Thus, Lemma \ref{le:quantile_bypass} with $S = g(X)$, $Z = f(X)$, $a = K\|X\|_\cA$ and $b = 4K^2\max_{A\in \cA}\|A\|$ gives
\begin{displaymath}
\p(|f(X) - \Med f(X)| \ge t) \le 2\exp\Big(-\min\Big(\frac{t^2}{CK^2\|X\|_\cA^2},\frac{t}{CK^2\max_{A\in \cA}\|A\|}\Big)\Big).
\end{displaymath}
Since the above inequality holds for arbitrary $t > 0$, we can use Lemma \ref{le:median} to complete the proof.
\end{proof}

\section{Application. Concentration inequalities for the empirical covariance operator \label{sec:covariance}}

Let us conclude with an application of  Theorem \ref{thm:uniform_Hanson-Wright} in the Gaussian setting, by providing a new proof of the concentration inequality for empirical approximations of the covariance operator of a Banach space valued random variable, proved recently in \cite{2014arXiv1405.2468K} by other methods. Since this part serves mostly as an illustration of applicability of Theorem \ref{thm:uniform_Hanson-Wright}, we do not present the general setting and motivation for this type of results, referring the Reader to the original paper \cite{2014arXiv1405.2468K}.

In the formulation of the following theorem we use $\|\cdot\|$ to denote both a norm of a vector in a Banach space and the operator norm.
\begin{theorem}\label{thm:Koltchinskii-Lounici}
Let $G$ be a Gaussian vector with values in a separable Banach space $E$ and let $\Sigma\colon E^\ast \to E$ be its covariance operator, i.e.
\begin{displaymath}
\Sigma u = \E\langle G,u\rangle G.
\end{displaymath}
Let $G_1,\ldots,G_n$ be i.i.d. copies of $G$ and define $\hat{\Sigma} \colon E^\ast\to E$ with the formula
\begin{displaymath}
\hat{\Sigma} u = \frac{1}{n}\sum_{k=1}^n \langle G_k,u\rangle G_k,\, u\in E^\ast.
\end{displaymath}
Then, for any $t \ge 1$,
\begin{displaymath}
\p\Big(\Big|\|\hat{\Sigma} - \Sigma\| - \E\|\hat{\Sigma} - \Sigma\|\Big| \ge C\|\Sigma\|\Big(1+\sqrt{\frac{r(\Sigma)}{n}}\Big)\sqrt{\frac{t}{n}} + \|\Sigma\|\frac{t}{n}\Big)\le e^{-t},
\end{displaymath}
where $r(\Sigma) = \frac{(\E\|G\|)^2}{\|\Sigma\|}$.
\end{theorem}

\begin{proof}
By the Karhunen-Lo\`eve theorem, there exists a sequence $x_k \in E$, such that almost surely
\begin{displaymath}
G = \sum_{j=1}^\infty x_j g_j,
\end{displaymath}
where $g_j$ are i.i.d. standard Gaussian variables. Let $\{g_{ij}\}_{1\le i\le n,j \in \N}$ be an array of i.i.d. standard Gaussian variables. We can assume that
\begin{displaymath}
G_i = \sum_{j=1}^\infty x_j g_{ij}.
\end{displaymath}
Then
\begin{displaymath}
\hat{\Sigma}u = \frac{1}{n}\sum_{k=1}^n\sum_{i=1}^\infty \sum_{j=1}^\infty \langle x_i,u\rangle x_j g_{ki}g_{kj}
\end{displaymath}
and
\begin{displaymath}
\Sigma u = \sum_{j=1}^\infty \langle x_j,u\rangle x_j.
\end{displaymath}
Therefore, denoting by $B^\ast$ the unit ball of $E^\ast$, we get
\begin{displaymath}
\|\hat{\Sigma} - \Sigma\| = \sup_{u,v \in B^\ast} \Big(\frac{1}{n}\sum_{k=1}^n\sum_{i=1}^\infty\sum_{j=1}^\infty \langle x_i,u\rangle\langle x_j, v\rangle g_{ki}g_{kj} - \E \frac{1}{n}\sum_{k=1}^n\sum_{i=1}^\infty\sum_{j=1}^\infty \langle x_i,u\rangle\langle x_j, v\rangle g_{ki}g_{kj} \Big),
\end{displaymath}
which puts us in position to use Theorem \ref{thm:uniform_Hanson-Wright} with $\cA = \{[n^{-1}\langle x_i,u\rangle\langle x_j,v\rangle\ind{k=l}]_{(k,i),(l,j)}\colon u,v\in B^\ast\}$ and $X = (g_{ki})_{k\le n,i\le \infty}$ (we skip the standard details of approximation by finite dimensional vectors).

Let us estimate the parameters of Theorem \ref{thm:uniform_Hanson-Wright}. Using the fact that each $A \in \cA $ is a block matrix with blocks of the form $\frac{1}{n}(\langle x_i,u\rangle)_{i=1}^\infty\otimes
(\langle x_j,v\rangle)_{j=1}^\infty$, one easily gets that
\begin{align*}
 \sup_{A\in\cA} \|A\|= \frac{1}{n}\sup_{u,v \in B^\ast} \Big(\sum_{i=1}^\infty \langle x_i,u\rangle^2\Big)^{1/2}\Big(\sum_{j=1}^\infty \langle x_j,v\rangle^2\Big)^{1/2} = \frac{1}{n}\sup_{u \in B^\ast} \sum_{i=1}^\infty \langle x_i,u\rangle^2 = \frac{1}{n}\|\Sigma\|.
\end{align*}
Passing to $\|X\|_{\cA}$, we have
\begin{align}\label{eq:bound_on_|X|}
\|X\|_\cA \le \E\sup_{A \in \cA}|AX| + \E \sup_{A\in\cA}|A^TX|.
\end{align}
Now,
\begin{align}\label{eq:expectation}
\E \sup_{A\in \cA}|A^TX| = &\frac{1}{n}\E\sup_{u,v\in B^\ast}\Big(\sum_{k=1}^n\sum_{j=1}^\infty\langle x_j,v\rangle^2\Big(\sum_{i=1}^\infty\langle x_i,u\rangle g_{ki}\Big)^2\Big)^{1/2}\\
& = \frac{1}{n}\sup_{v\in B^\ast}\Big(\sum_{j=1}^\infty\langle x_j,v\rangle^2 \Big)^{1/2}\E\sup_{u \in B^\ast}\Big(\sum_{k=1}^n \Big(\sum_{i=1}^\infty\langle x_i,u\rangle g_{ki}\Big)^2\Big)^{1/2}\nonumber\\
& = \frac{1}{n}\|\Sigma\|^{1/2}\E\sup_{u \in B^\ast}\Big(\sum_{k=1}^n \Big(\sum_{i=1}^\infty\langle x_i,u\rangle g_{ki}\Big)^2\Big)^{1/2}\nonumber
\end{align}
To bound the last expectation, we can use the Gordon-Chevet inequality \cite{Chevet,MR800188}, which asserts that for any Banach spaces $E,F$ and points $x_i \in E$, $y_k \in F$, the random operator
\begin{displaymath}
\Gamma = \sum_{i,k} g_{ki} x_i\otimes y_k \colon E^\ast \to F,
\end{displaymath}
satisfies
\begin{align*}
\E \|\Gamma\|_{E^\ast\to F} \le& \sup\{\|\sum_i t_i x_i\|_E\colon \sum_i t_i^2 =1\}\E \|\sum_k g_k y_k\|_F \\
&+ \sup\{\|\sum_k t_k y_k\|_F\colon \sum_k t_k^2 =1\}\E\|\sum_i g_i x_i\|_E,
\end{align*}
where $g_i$'s are i.i.d. standard Gaussian variables.

Applying this inequality with $\Gamma = \sum_{k,i} g_{ki} x_i\otimes y_k \colon E^\ast \to \ell_2^n$,
where $y_1,\ldots,y_n$ is the standard basis of $\ell_2^n$, we get
\begin{align*}
&\E\sup_{u \in B^\ast}\Big(\sum_{k=1}^n \Big(\sum_{j=1}^\infty\langle x_i,u\rangle g_{ki}\Big)^2\Big)^{1/2} = \E \|\Gamma\|_{E^\ast \to \ell_2^n}\\
&\le \sup\{\|\sum_{i=1}^\infty t_i x_i\| \colon \sum_{i=1}^\infty t_i^2 =1\} \E |\sum_{k=1}^n g_k y_k| + \sup\{|\sum_{k=1}^n t_k y_k| \colon \sum_{k=1}^n t_k^2 =1\}\E\|\sum_{i=1}^\infty g_i x_i\|\\
& \le \sup_{u \in B^\ast} \Big(\sum_{i=1}^\infty \langle x_i,u\rangle^2\Big)^{1/2}\sqrt{n} + 1\cdot\E\|G\| = \|\Sigma\|^{1/2}\sqrt{n} + \E\|G\|.
\end{align*}
Going back to \eqref{eq:expectation}, we get
\begin{displaymath}
\E\sup_{A\in \cA} |A^T X|\le \frac{\|\Sigma\|}{\sqrt{n}} + \frac{\|\Sigma\|^{1/2}\E\|G\|}{n}.
\end{displaymath}
By symmetry, an analogous bound holds for the other expectation on the right-hand side of \eqref{eq:bound_on_|X|}, hence
\begin{displaymath}
\|X\|_\cA \le 2 \frac{\|\Sigma\|}{\sqrt{n}} + 2\frac{\|\Sigma\|^{1/2}\E\|X\|}{n} = 2\frac{\|\Sigma\|}{\sqrt{n}} + 2\frac{\|\Sigma\|}{\sqrt{n}} \sqrt{\frac{r(\Sigma)}{n}}
\end{displaymath}
Combining this with the estimate on $\sup_{A\in \cA}\|A\|$ and Theorem \ref{thm:uniform_Hanson-Wright}, we get for $t \ge 1$,
\begin{displaymath}
\p\Big(\Big|\|\hat{\Sigma} - \Sigma\| - \E\|\hat{\Sigma} - \Sigma\|\Big| \ge C\|\Sigma\|\Big(1+\sqrt{\frac{r(\Sigma)}{n}}\Big)\sqrt{\frac{t}{n}} + \|\Sigma\|\frac{t}{n}\Big)\le e^{-t},
\end{displaymath}
which ends the proof.
\end{proof}

\bibliographystyle{abbrv}
\bibliography{citations}
\end{document}